\documentclass[reqno,12pt]{amsart}
\usepackage{amsmath, epsfig, cite}
\usepackage{amssymb}
\usepackage{amsfonts}
\usepackage{latexsym}
\usepackage{amsthm}

\newtheorem{theorem}{Theorem}%[section]

\newtheorem{conjecture}{Conjecture}
\newtheorem{lemma}{Lemma}

\theoremstyle{remark}

\newcommand{\N}{\mathbb N}

\allowdisplaybreaks

\author{Victor J.\ W.\ Guo}
\address{School of Mathematical Sciences, Huaiyin Normal University,
Huai'an 223300, Jiangsu, People's Republic of China}
\email{jwguo@hytc.edu.cn}
\thanks{The first author was partially supported by the National Natural
Science Foundation of China (grant 11771175).}

\author{Michael J.\ Schlosser}
\address{Fakult\"at f\"ur Mathematik, Universit\"at Wien,
Oskar-Morgenstern-Platz~1, A-1090 Vienna, Austria}
\email{michael.schlosser@univie.ac.at}
\title[Some new $q$-congruences for truncated basic series]{Some
new $q$-congruences for truncated\\ basic hypergeometric series}



\begin{document}

\begin{abstract}
We provide several new $q$-congruences for truncated basic
hypergeometric series, mostly of arbitrary order. Our results
include congruences modulo the square or the cube of a cyclotomic
polynomial, and in some instances, parametric generalizations
thereof. These are established by a variety of techniques including
polynomial argument, creative microscoping (a method recently
introduced by the first author in collaboration with Zudilin),
Andrews' multiseries generalization of the Watson transformation,
and induction. We also give a number of related conjectures
including congruences modulo the fourth power of a cyclotomic
polynomial.
\end{abstract}

\maketitle

\section{Introduction}
In 1914, Ramanujan \cite{Ramanujan} stated rather mysteriously a number
of formulas for $1/\pi$, including
\begin{equation*}
\sum_{k=0}^\infty \frac{(\frac{1}{2})_k^3}{k!^3 } (6k+1)\frac{1}{4^k}
=\frac{4}{\pi}. %\label{eq:ram}
\end{equation*}
In 1997, Van Hamme \cite{Hamme} conjectured 13 interesting $p$-adic
analogues of Ramanujan's or Ramanujan-type formulas for $1/\pi$, such as
\begin{equation}\label{eq:p4}
\sum_{k=0}^{\frac{p-1}2} \frac{(\frac{1}{2})_k^3}{k!^3 } (6k+1)\frac{1}{4^k}
\equiv p(-1)^{\frac{p-1}2}\pmod{p^4},
\end{equation}
where $(a)_n=a(a+1)\cdots(a+n-1)$ denotes the Pochhammer symbol and $p$
is an odd prime. All of the 13 supercongruences
have been confirmed by different techniques up to now (see \cite{OZ,Swisher}).
For some informative background on Ramanujan-type supercongruences,
see Zudilin's paper \cite{Zud2009}.
During the past few years, $q$-analogues of congruences and supercongruences
have caught the interests of many authors (see, for example,
\cite{Gorodetsky,Guillera3,Guo2018,Guo-i,Guo-t,Guo-ef,Guo-gz,Guo-par,
Guo-m=d,GPZ,GJZ,GS,GS2,
GW,GZ14,GuoZu,GuoZu2,NP,PS,Straub,Tauraso1,Tauraso2,Zudilin}).
As made explicit in \cite{GuoZu}, $q$-supercongruences are related to
studying the asymptotic behaviour of $q$-series at roots of unity.
This hints towards an intrinsic connection to mock theta functions
and quantum modular forms (see e.g. \cite{FOR,Zag}).

Congruences of truncated hypergeometric series
modulo a high power of a prime such as in Equation \eqref{eq:p4} are special.
Similarly, in the setting of truncated basic hypergeometric series,
congruences modulo some power of a cyclotomic polynomial are special and,
already for the exponent being $\geqslant 2$, are
typically difficult to prove.

Recently, the first author \cite[Theorem 1.1]{Guo-m=d} proved that
for $n\equiv 4\pmod{5}$
\begin{equation*}
\sum_{k=0}^{n-1}\frac{(q;q^5)_k^5}{(q^5;q^5)_k^5}q^{5k}\equiv
0\pmod{\Phi_n(q)^2},
\end{equation*}
which, under the substitution $q\mapsto q^{-1}$, can be written as
\begin{equation*}
\sum_{k=0}^{n-1}\frac{(q;q^5)_k^5}{(q^5;q^5)_k^5}q^{15k}\equiv
0\pmod{\Phi_n(q)^2}.
\end{equation*}
It follows that for $n\equiv 4\pmod{5}$
\begin{equation}
\sum_{k=0}^{n-1}[10k+1]\frac{(q;q^5)_k^5}{(q^5;q^5)_k^5}q^{5k}\equiv
0\pmod{\Phi_n(q)^2}. \label{eq:guo}
\end{equation}
Here and in what follows, we adopt the standard $q$-notation:
$(a;q)_n=(1-a)(1-aq)\cdots (1-aq^{n-1})$ is the {\em $q$-shifted factorial};
$(a_1,a_2,\ldots,a_m;q)_n=(a_1;q)_n (a_2;q)_n\cdots (a_m;q)_n$
is a product of $q$-shifted factorials;
$[n]=[n]_q=1+q+\cdots+q^{n-1}$ is the {\em $q$-integer};
and $\Phi_n(q)$ denotes the $n$-th {\em cyclotomic polynomial} in $q$
(see \cite{Lang}), which may be defined as
\begin{align*}
\Phi_n(q)=\prod_{\substack{1\leqslant k\leqslant n\\ \gcd(n,k)=1}}(q-\zeta^k),
\end{align*}
where $\zeta$ is an $n$-th primitive root of unity.

We find that for $n\equiv 2\pmod{5}$ the $q$-congruence \eqref{eq:guo}
even holds modulo $\Phi_n(q)^3$. More generally we are able to extend
\eqref{eq:guo} to the following infinite family of $q$-congruences.
%It turns out that much more is true. The $q$-shifted factorial $(q;q^5)_k$ can be generalized to $(q;q^d)_k$ for $d\geqslant 5$, and
%the corresponding congruence sometimes even holds modulo $\Phi_n(q)^3$. Namely, we shall prove the following stronger result.
\begin{theorem}\label{thm:main-1}
Let $d\geqslant 5$ be an odd integer. Then
\begin{equation}
\sum_{k=0}^{n-1}[2dk+1]\frac{(q;q^d)_k^d}{(q^d;q^d)_k^d}q^{\frac{d(d-3)k}{2}}
\equiv
\begin{cases} 0\pmod{\Phi_n(q)^2}, &\text{if $n\equiv -1\pmod{d}$,}\\[5pt]
0\pmod{\Phi_n(q)^3}, &\text{if $n\equiv -\frac{1}{2}\pmod{d}$.}
\end{cases}  \label{eq:main-1}
\end{equation}
\end{theorem}

Note that for $d\geqslant 7$ and $n\equiv -1\pmod{d}$ the above $q$-congruence
cannot be deduced directly from \cite[Theorem 1.1]{Guo-m=d} in the same way
as the $q$-congruence \eqref{eq:guo} is derived.
This is because the arguments $q^d$ and $q^{\frac{d(d-3)}{2}}$
are different for $d\geqslant 7$.
It should be pointed out that the $q$-congruence \eqref{eq:main-1} does not
hold for $d=3$.
Like many results given in \cite{GS}, Theorem~\ref{thm:main-1} has a
companion as follows.
\begin{theorem}\label{thm:main-2}
Let $d\geqslant 3$ be an odd integer and let $n>1$. Then
\begin{equation}
\sum_{k=0}^{n-1}[2dk-1]\frac{(q^{-1};q^d)_k^d}{(q^d;q^d)_k^d}
q^{\frac{d(d-1)k}{2}}
\equiv
\begin{cases} 0\pmod{\Phi_n(q)^2}, &\text{if $n\equiv 1\pmod{d}$,}\\[5pt]
0\pmod{\Phi_n(q)^3}, &\text{if $n\equiv \frac{1}{2}\pmod{d}$.}
\end{cases} \label{eq:main-2}
\end{equation}
\end{theorem}

We shall also prove the following result, which was originally conjectured
by the first author \cite[Conjecture~1.3]{Guo-m=d} who provided a proof
of the modulus $[n]\Phi_n(q)$ case~\cite[Theorem~1.2]{Guo-m=d}.
\begin{theorem} \label{thm:main-3}
Let $n>1$ be an odd integer. Then
\begin{align}
\sum_{k=0}^{n-1}\frac{(q^{-1};q^2)_k^2}{(q^2;q^2)_k^2} q^{2k}
&\equiv 0 \pmod{[n]^2}, \label{eq:main-3-1} \\\intertext{and}
\sum_{k=0}^{\frac{n+1}{2}}\frac{(q^{-1};q^2)_k^2}{(q^2;q^2)_k^2} q^{2k}
&\equiv 0 \pmod{[n]^2}. \label{eq:main-3-2}
\end{align}
\end{theorem}

As mentioned in \cite{Guo-m=d}, there are many similar congruences
modulo $\Phi_n(q)^2$ for truncated basic hypergeometric series.
In this paper, we shall give more such examples (theorems or conjectures).
The simplest example is as follows.
\begin{theorem}\label{thm:main-4}
Let $n\geqslant 4$ be an integer with $\gcd(n,3)=1$. Then
\begin{equation}
\sum_{k=0}^{n-1}\frac{(q^{-1},q^{-2};q^3)_k }{(q^3;q^3)_k^2}q^{3k}\equiv
0\pmod{\Phi_n(q)^2}.
\end{equation}
\end{theorem}

Note that the first author, in joint work with Pan and Zhang \cite{GPZ},
proved that for any odd integer $n\geqslant 5$ with $\gcd(n,3)=1$
there holds
\begin{equation*}
\sum_{k=0}^{n-1}\frac{(q,q^{2};q^3)_k}{(q^3;q^3)_k^2}q^{3k}\equiv
\left(\frac{n}{3}\right) q^{\frac{n^2-1}{3}}\pmod{\Phi_n(q)^2},
\end{equation*}
where $\left(\frac{\cdot}{3}\right)$ denotes the Legendre symbol modulo $3$.
This $q$-congruence was originally conjectured in \cite{GZ14} when $n=p$
is an odd prime.

We shall prove Theorems~\ref{thm:main-1} and \ref{thm:main-2}
in Sections~\ref{sec:prth1} and \ref{sec:prth2} by using the
{\it creative microscoping} method developed by the first author and
Zudilin \cite{GuoZu}.
We prove these by first establishing their parametric
generalizations modulo $(1-aq^n)(a-q^n)$ and then letting $a\to 1$.
The proofs are similar to that of \cite[Theorem~1.1]{Guo-m=d} but
also require Andrews' multiseries generalization of the Watson
transformation \cite[Theorem~4]{Andrews75} (which was already used
by the first author, Jouhet and Zeng~\cite{GJZ}
for proving some $q$-analogues of Calkin's congruence~\cite{Calkin}).
It is worth mentioning that we need to add the parameter $a$ and
also its powers in many places of the left-hand sides of \eqref{eq:main-1}
and \eqref{eq:main-2} in order to establish the desired generalizations
modulo $(1-aq^n)(a-q^n)$.
Therefore, the proofs of Theorems~\ref{thm:main-1} and \ref{thm:main-2}
are quite different from those in the recent two joint papers of
us \cite{GS,GS2}, where the parameter $a$
is inserted in a more natural way (without $a^2$ and higher powers of $a$)
as done in \cite{GuoZu}.
The proofs of Theorems~\ref{thm:main-3} and \ref{thm:main-4}
are based on two $q$-series identities and are given in
Sections~\ref{sec:prth3} and \ref{sec:prth4}, respectively.
Two more congruences modulo $\Phi_n(q)^2$ are given in Section~\ref{sec:more}.
We give some related conjectures in the final Section~\ref{sec:rems}.
These include two refinements of Theorems~\ref{thm:main-1} and
\ref{thm:main-2}, some extensions of Theorem~\ref{thm:main-4}
for $n\equiv 1\pmod{3}$, and similar conjectures.

%%%%%%%%%%%%%%%%%%%%%%%%%%%%%%%%%%%%%%%%%%%%%%%%%%%%%%%%%%%%%%%%%%%%%%%%%%%%%%%%%%%%%%%%%%%%%%%%%%%%%%%%%%%%%%%%%%%%%%%%%%%%%%%%%%%%%%%%%%%%%%%%%%%%%%%%%%%%%%%%%%%%%%%%%%%%%%%%%%%%%%%%%%%%%%%%%%%%%%%%%%%%%%%%%%%%%%%%%%%%%
\section{Proof of Theorem \ref{thm:main-1}}\label{sec:prth1}
We first establish the following parametric generalization of
Theorem~\ref{thm:main-1} for the case $n\equiv -1\pmod{d}$.
\begin{theorem}\label{thm:a-1}
Let $d\geqslant 5$ be an odd integer and let $n\equiv -1\pmod{d}$.
Then modulo $(1-aq^n)(a-q^n)$,
\begin{align}
\sum_{k=0}^{n-1}&[2dk+1]\frac{(a^{d-1}q, a^{d-3}q,\ldots, a^2q;q^d)_k}
{(a^{d-2}q^d, a^{d-4}q^d,\ldots,aq^d;q^d)_k} \notag\\
&\times\frac{
(a^{1-d}q, a^{3-d}q,\ldots, a^{-2}q;q^d)_k (q;q^d)_k}
{(a^{2-d}q^d, a^{4-d}q^d,\ldots,a^{-1}q^d;q^d)_k (q^d;q^d)_k }
 q^{\frac{d(d-3)k}{2}}
\equiv 0.\label{eq:a-1}
\end{align}
\end{theorem}
\begin{proof} It is clear that $\gcd(d,n)=1$ and therefore the numbers
$d,2d,\ldots, (n-1)d$ are all not divisible by $n$. This implies that the
denominators of the left-hand side of \eqref{eq:a-1} do not contain
the factor $1-aq^n$ nor $1-a^{-1}q^n$.
Thus, for $a=q^{-n}$ or $a=q^n$, the left-hand side of \eqref{eq:a-1}
can be written as
\begin{align}
\sum_{k=0}^{\frac{dn-n-1}{d}}&[2dk+1]\frac{(q^{1-(d-1)n}, q^{1-(d-3)n},\ldots,
q^{1-2n};q^d)_k  }
{(q^{d-(d-2)n}, q^{d-(d-4)n},\ldots, q^{d-n};q^d)_k  } \notag\\
&\times \frac{(q^{(d-1)n+1}, q^{(d-3)n+1},\ldots,
q^{2n+1};q^d)_k (q;q^d)_k}
{(q^{(d-2)n+d}, q^{(d-4)n+d},\ldots, q^{n+d};q^d)_k (q^d;q^d)_k }
q^{\frac{d(d-3)k}{2}},   \label{eq:a-2}
\end{align}
where we have used $(q^{1-(d-1)n};q^d)_k=0$ for $k>(dn-n-1)/d$.

Let
\begin{equation*}
\begin{bmatrix}n\\k\end{bmatrix}=\begin{bmatrix}n\\k\end{bmatrix}_q=
\frac{(q;q)_n}{(q;q)_k(q;q)_{n-k}}
\end{equation*}
be the {\it $q$-binomial coefficient}.
It is easy to see,
with $\binom k2=k(k-1)/2$ denoting a binomial coefficient, that
\begin{align}
\frac{(q^{1-(d-1)n};q^d)_k}{(q^d;q^d)_k} q^{\frac{d(d-3)k}{2}}
&=(-1)^k \begin{bmatrix}(dn-n-1)/d\\k\end{bmatrix}_{q^d}
q^{d\binom k2+\left(n+1-dn+\frac{d(d-3)}{2}\right)k},
\label{qdk-0}\\[5pt]
\frac{(q^{1-(d-3)n};q^d)_k}{(q^{d-(d-2)n};q^d)_k}
&=\frac{(q^{d-(d-2)n+dk};q^d)_{(n+1-d)/d}}{(q^{d-(d-2)n};q^d)_{(n+1-d)/d}},\notag\\[5pt]
\frac{(q^{1-(d-5)n};q^d)_k}{(q^{d-(d-4)n};q^d)_k}
&=\frac{(q^{d-(d-4)n+dk};q^d)_{(n+1-d)/d}}{(q^{d-(d-4)n};q^d)_{(n+1-d)/d}},\notag\\[5pt]
&\ \  \vdots \notag \\[5pt]
\frac{(q^{1-2n};q^d)_k}{(q^{d-3n};q^d)_k}
&=\frac{(q^{d-3n+dk};q^d)_{(n+1-d)/d}}{(q^{d-3n};q^d)_{(n+1-d)/d}}, \notag
\end{align}
and
\begin{align*}
\frac{(q^{(d-1)n+1};q^d)_k}{(q^{(d-2)n+d};q^d)_k}
&=\frac{(q^{(d-2)n+dk+d};q^d)_{(n+1-d)/d}}{(q^{(d-2)n+d};q^d)_{(n+1-d)/d}},\\[5pt]
\frac{(q^{(d-3)n+1};q^d)_k}{(q^{(d-4)n+d};q^d)_k}
&=\frac{(q^{(d-4)n+dk+d};q^d)_{(n+1-d)/d}}{(q^{(d-4)n+d};q^d)_{(n+1-d)/d}},\\[5pt]
&\ \  \vdots  \\[5pt]
\frac{(q^{2n+1};q^d)_k}{(q^{n+d};q^d)_k}
&=\frac{(q^{2n+dk+d};q^d)_{(n+1-d)/d}}{(q^{n+d};q^d)_{(n+1-d)/d}},\\[5pt]
\frac{(q;q^d)_k}{(q^{d-n};q^d)_k}
&=\frac{(q^{d-n+dk};q^d)_{(n+1-d)/d} }{(q^{d-n};q^d)_{(n+1-d)/d} }.
\end{align*}
Note that the right-hand sides of the identities after \eqref{qdk-0}
are all polynomials in $q^{dk}$, of degree $(n+1-d)/d$ in the first group, and
of degree $(n+1-d)/d$ in the second one too. Moreover,
\begin{align*}
&d\binom k2+\left(n+1-dn+\frac{d(d-3)}{2}\right)k  \\[5pt]
&=d\binom{(dn-n-1)/d-k}2-d\binom{(dn-n-1)/d}2+\frac{d(d-5)k}{2}.
\end{align*}
Therefore, we can write \eqref{eq:a-2} in the following form
\begin{equation}
\sum_{k=0}^{\frac{dn-n-1}{d}}(-1)^k q^{d\binom{(dn-n-1)/d-k}2}
 \begin{bmatrix}(dn-n-1)/d\\k\end{bmatrix}_{q^d}P(q^{dk}), \label{eq:a-p}
\end{equation}
where $P(q^{dk})$ is a polynomial in $q^{dk}$ of degree
$2+(n+1-d)(d-1)/d+(d-5)/2=(dn-n-1)/d-(d-3)/2\leqslant (dn-n-1)/d-1$.

Recall that the $q$-binomial theorem (see \cite[p.~36]{Andrews})
can be written as
\begin{equation*}
\sum_{k=0}^{n}(-1)^k  \begin{bmatrix}n\\k\end{bmatrix}q^{\binom k2} z^k=(z;q)_n.
\end{equation*}
Putting $z=q^{-j}$ in the above identity and replacing $k$ with $n-k$, we get
\begin{equation}
\sum_{k=0}^{n}(-1)^k  \begin{bmatrix}n\\k\end{bmatrix}q^{\binom{n-k}2+jk}=
0\quad\text{for}\ 0\leqslant j\leqslant n-1,  \label{eq:qbino}
\end{equation}
which immediately means that the expression in \eqref{eq:a-2},
which equals \eqref{eq:a-p}, vanishes.
This proves \eqref{eq:a-1}.
\end{proof}

In order to prove Theorem \ref{thm:main-1} for the case
$n\equiv -\frac{1}{2}\pmod{d}$, we need the following lemma.
\begin{lemma}\label{lem:2.2}
Let $d\geqslant 3$ be an odd integer and let $n\equiv -\frac{1}{2}\pmod{d}$.
Then for $0\leqslant k\leqslant (dn-2n-1)/d$, modulo $\Phi_n(q)$ we have
\begin{equation*}
\frac{(aq;q^d)_{(dn-2n-1)/d-k}}{(q^d/a;q^d)_{(dn-2n-1)/d-k}}
\equiv (-a)^{(dn-2n-1)/d-2k}\frac{(aq;q^d)_k}{(q^d/a;q^d)_k}
q^{(dn-2n-d+1)(dn-2n-1)/(2d)+(d-1)k}.
\end{equation*}
\end{lemma}
\begin{proof}Since $q^n\equiv 1\pmod{\Phi_n(q)}$, we have
\begin{align}\label{aqcong}
\frac{(aq;q^d)_{(dn-2n-1)/d}}{(q^d/a;q^d)_{(dn-2n-1)/d}}
&=\frac{(1-aq)(1-aq^{d+1})\cdots (1-aq^{dn-2n-d})}
{(1-q^d/a)(1-q^{2d}/a)\cdots (1-q^{dn-2n-1}/a)} \notag\\[5pt]
&\equiv \frac{(1-aq)(1-aq^{d+1})\cdots (1-aq^{dn-2n-d})}
{(1-q^{d+2n-dn}/a)(1-q^{2d+2n-dn}/a)\cdots (1-q^{-1}/a)}\notag\\[5pt]
&=(-a)^{(dn-2n-1)/d}q^{(dn-2n-d+1)(dn-2n-1)/(2d)} \pmod{\Phi_n(q)}.
\end{align}
Furthermore, modulo $\Phi_n(q)$, there holds
\begin{align*}
&\frac{(aq;q^d)_{(dn-2n-1)/d-k}}{(q^d/a;q^d)_{(dn-2n-1)/d-k}} \\[5pt]
&=\frac{(aq;q^d)_{(dn-2n-1)/d}}{(q^d/a;q^d)_{(dn-2n-1)/d}}
\frac{(1-q^{dn-2n+d-1-dk}/a)(1-q^{dn-2n+2d-1-dk}/a)\cdots (1-q^{dn-2n-1}/a)}
{(1-aq^{dn-2n-dk})(1-aq^{dn-2n+d-dk})\cdots (1-aq^{dn-2n-d})}
\\[5pt]
&\equiv \frac{(aq;q^d)_{(dn-2n-1)/d}}{(q^d/a;q^d)_{(dn-2n-1)/d}}
\frac{(1-q^{d-1-dk}/a)(1-q^{2d-1-dk}/a)\cdots (1-q^{-1}/a)}
{(1-aq^{-dk})(1-aq^{d-dk})\cdots (1-aq^{-d})},
\end{align*}
which together with \eqref{aqcong} establishes the assertion.
\end{proof}

We now give a parametric generalization of Theorem \ref{thm:main-1}
for the case $n\equiv -\frac{1}{2}\pmod{d}$.
\begin{theorem}\label{thm:a-2}
Let $d\geqslant 3$ be an odd integer and let $n\equiv -\frac{1}{2}\pmod{d}$.
Then modulo $\Phi_n(q)(1-aq^n)(a-q^n)$,
\begin{align}
\sum_{k=0}^{n-1}&[2dk+1]\frac{(a^{d-2}q, a^{d-4}q,\ldots, aq;q^d)_k}
{(a^{d-2}q^d, a^{d-4}q^d,\ldots,aq^d;q^d)_k}\notag\\ &\times
\frac{(a^{2-d}q, a^{4-d}q,\ldots, a^{-1}q;q^d)_k (q;q^d)_k}
{(a^{2-d}q^d, a^{4-d}q^d,\ldots,
a^{-1}q^d;q^d)_k (q^d;q^d)_k }  q^{\frac{d(d-3)k}{2}}
\equiv 0.\label{eq:aa-2}
\end{align}
\end{theorem}

\begin{proof}
By Lemma~\ref{lem:2.2}, for $0\leqslant k\leqslant (dn-2n-1)$,
we can check that the $k$-th and $((dn-2n-1)/d-k)$-th terms on the
left-hand side of \eqref{eq:aa-2} modulo $\Phi_n(q)$ cancel each other.
Moreover, for $(dn-2n-1)/d<k\leqslant n-1$, the $q$-shifted factorial
$(q;q^d)_k$ contains the factor $1-q^n$ and is therefore divisible by
$\Phi_n(q)$. This proves that the congruence \eqref{eq:aa-2} is true
modulo $\Phi_n(q)$.

To prove that \eqref{eq:aa-2} is also true modulo $(1-aq^n)(a-q^n)$,
it suffices to prove the following identity:
\begin{align}
\sum_{k=0}^{\frac{dn-2n-1}{d}}&[2dk+1]\frac{(q^{1-(d-2)n}, q^{1-(d-4)n},\ldots,
q^{1-n};q^d)_k  }
{(q^{d-(d-2)n}, q^{d-(d-4)n},\ldots, q^{d-n};q^d)_k  } \notag\\
&\times \frac{(q^{(d-2)n+1}, q^{(d-4)n+1},\ldots,
q^{n+1};q^d)_k (q;q^d)_k}
{(q^{(d-2)n+d}, q^{(d-4)n+d},\ldots, q^{n+d};q^d)_k (q^d;q^d)_k } q^{\frac{d(d-3)k}{2}}
=0,   \label{eq:a-3}
\end{align}
where we have used that $(q^{1-(d-2)n};q^d)_k=0$ for $k>(dn-2n-1)/d$.
This time the method employed to prove \eqref{eq:a-2} does not work. Instead,
we shall use Andrews' multiseries generalization of the Watson transformation
\cite[Theorem~4]{Andrews75}:
\begin{align}
&\sum_{k\geqslant 0}\frac{(a,q\sqrt{a},-q\sqrt{a},b_1,c_1,\dots,b_m,c_m,q^{-N};q)_k}
{(q,\sqrt{a},-\sqrt{a},aq/b_1,aq/c_1,\dots,aq/b_m,aq/c_m,aq^{N+1};q)_k}
\left(\frac{a^mq^{m+N}}{b_1c_1\cdots b_mc_m}\right)^k \notag\\[5pt]
&\quad=\frac{(aq,aq/b_mc_m;q)_N}{(aq/b_m,aq/c_m;q)_N}
\sum_{l_1,\dots,l_{m-1}\geqslant 0}
\frac{(aq/b_1c_1;q)_{l_1}\cdots(aq/b_{m-1}c_{m-1};q)_{l_{m-1}}}
{(q;q)_{l_1}\cdots(q;q)_{l_{m-1}}} \notag\\[5pt]
&\quad\quad\times\frac{(b_2,c_2;q)_{l_1}\dots(b_m,c_m;q)_{l_1+\dots+l_{m-1}}}
{(aq/b_1,aq/c_1;q)_{l_1}
\dots(aq/b_{m-1},aq/c_{m-1};q)_{l_1+\dots+l_{m-1}}} \notag\\[5pt]
&\quad\quad\times\frac{(q^{-N};q)_{l_1+\dots+l_{m-1}}}
{(b_mc_mq^{-N}/a;q)_{l_1+\dots+l_{m-1}}}
\frac{(aq)^{l_{m-2}+\dots+(m-2)l_1} q^{l_1+\dots+l_{m-1}}}
{(b_2c_2)^{l_1}\cdots(b_{m-1}c_{m-1})^{l_1+\dots+l_{m-2}}}.  \label{andrews}
\end{align}

Let $q\mapsto q^d$, $a=q$, $b_1=q^{(d+1)/2}$, $m=(d-1)/2$, and
$N=(dn-2n-1)/d$ in \eqref{andrews}.
Moreover, put
\begin{equation*}
\{c_1,b_2,c_2,\ldots,b_m,c_m\}=
\{q^{1-(d-4)n},q^{1-(d-6)n},\ldots,q^{1-n}, q^{(d-2)n+1}, q^{(d-4)n+1},\ldots, q^{n+1}\}
\end{equation*}
with $b_m=q^{1-n}$ and $c_m=q^{3n+1}$. Then the left-hand side of
\eqref{andrews} reduces to the left-hand side of \eqref{eq:a-3}, while
the right-hand side of \eqref{andrews} contains the factor
\begin{equation*}
(q^{d+1}/b_m c_m;q^d)_N
=(q^{d-2n-1};q^d)_N=0,
\end{equation*}
because $d-2n-1\equiv 0\pmod{d}$, $d-2n-1\leqslant 0$ and $N>(d-2n-1)/d$.
This proves \eqref{eq:a-3}, i.e.\
the congruence \eqref{eq:aa-2} holds modulo $(1-aq^n)(a-q^n)$.
Since the polynomials $\Phi_n(q)$ and $(1-aq^n)(a-q^n)$
are clearly relatively prime, the proof of \eqref{eq:aa-2} is complete.
\end{proof}

\begin{proof}[Proof of Theorem {\rm\ref{thm:main-1}}]
For $n\equiv -1\pmod{d}$, the limits of the denominators in \eqref{eq:a-1}
as $a\to1$ are relatively prime to $\Phi_n(q)$.
On the other hand, the limit of $(1-aq^n)(a-q^n)$ as $a\to1$ has the factor
$\Phi_n(q)^2$.
It follows that the limiting case $a\to 1$ of the congruence \eqref{eq:a-1}
reduces to \eqref{eq:main-1} for the case $n\equiv -1\pmod{d}$.

Similarly, for $n\equiv -\frac{1}{2}\pmod{d}$, the limit of
$\Phi_n(q)(1-aq^n)(a-q^n)$ as $a\to1$ has the factor $\Phi_n(q)^3$,
and so the limiting case $a\to 1$ of the congruence \eqref{eq:aa-2}
reduces to \eqref{eq:main-1} for the case $n\equiv -\frac{1}{2}\pmod{d}$.
This completes the proof of the theorem.
\end{proof}

%%%%%%%%%%%%%%%%%%%%%%%%%%%%%%%%%%%%%%%%%%%%%%%%%%%%%%%%%%%%%%%%%%%%%%%%%%%%%%%%%%%%%%%%%%%%%%%%%%%%%%%%%%%%%%%%%%%%%%%%%%%%%%%%%%%%%%%%%%%%%%%%%%%%%%%%%%%%%%%%%%%%%%%%%%%%%%%%%%%%%%%%%%%%%%%%%%%%%%%%%%%%%%%%%%%%%%%%%%
\section{Proof of Theorem \ref{thm:main-2}}\label{sec:prth2}
The proof of Theorem \ref{thm:main-2} is similar to that of
Theorem~\ref{thm:main-1}. We have the following parametric generalization
of Theorem~\ref{thm:main-2} for the case $n\equiv -1\pmod{d}$. Its proof is
completely analogous to that of Theorem~\ref{thm:a-1} and is left to the
interested reader.
\begin{theorem}\label{thm:main-2-one}
Let $d\geqslant 3$ be an odd integer and let $n\equiv 1\pmod{d}$.
Then modulo $(1-aq^n)(a-q^n)$,
\begin{align*}
&\sum_{k=0}^{n-1}[2dk-1]\frac{(a^{d-1}q^{-1}, a^{d-3}q^{-1},\ldots, a^2q^{-1};q^d)_k}
{(a^{d-2}q^d, a^{d-4}q^d,\ldots,aq^d;q^d)_k   }  \\[5pt]
&\quad\quad\times\frac{(a^{1-d}q^{-1}, a^{3-d}q^{-1},\ldots,
a^{-2}q^{-1};q^d)_k (q^{-1};q^d)_k}
{(a^{2-d}q^d, a^{4-d}q^d,\ldots, a^{-1}q^d;q^d)_k (q^d;q^d)_k }
q^{\frac{d(d-1)k}{2}}\equiv 0.
\end{align*}
\end{theorem}

Moreover, we have the following result similar to Lemma~\ref{lem:2.2}.
\begin{lemma}\label{lem:3.2}
Let $d$ be a positive odd integer and let $n\equiv \frac{1}{2}\pmod{d}$,
Then for $0\leqslant k\leqslant (dn-2n+1)/d$, modulo $\Phi_n(q)$, we have
\begin{equation*}
\frac{(aq^{-1};q^d)_{(dn-2n+1)/d-k}}{(q^d/a;q^d)_{(dn-2n+1)/d-k}}
\equiv (-a)^{(dn-2n+1)/d-2k}\frac{(aq^{-1};q^d)_k}{(q^d/a;q^d)_k}
q^{(dn-2n-d-1)(dn-2n+1)/(2d)+(d+1)k}.
\end{equation*}
\end{lemma}

By Lemma \ref{lem:3.2} and Andrews' transformation \eqref{andrews},
we can establish the following parametric generalization
of Theorem \ref{thm:main-2} for the case $n\equiv \frac{1}{2}\pmod{d}$.
\begin{theorem}\label{thm:main-2-two}
Let $d\geqslant 3$ be an odd integer and let $n\equiv \frac{1}{2}\pmod{d}$.
Then modulo $\Phi_n(q)(1-aq^n)(a-q^n)$,
\begin{align*}
&\sum_{k=0}^{n-1}[2dk-1]\frac{(a^{d-2}q^{-1}, a^{d-4}q^{-1},\ldots, aq^{-1};q^d)_k}
{(a^{d-2}q^d, a^{d-4}q^d,\ldots,aq^d;q^d)_k   }  \\[5pt]
&\quad\quad\times\frac{(a^{2-d}q^{-1}, a^{4-d}q^{-1},\ldots,
a^{-1}q^{-1};q^d)_k (q^{-1};q^d)_k }
{(a^{2-d}q^d, a^{4-d}q^d,\ldots, a^{-1}q^d;q^d)_k (q^d;q^d)_k }
q^{\frac{d(d-1)k}{2}} \equiv 0.
\end{align*}
\end{theorem}

The proof of Theorem \ref{thm:main-2} then follows from
Theorems~\ref{thm:main-2-one} and \ref{thm:main-2-two}
by taking the limit $a\to 1$.

Finally, we point out that for $d=3$ and any $n>0$ the sum in
Theorem~\ref{thm:main-2-two} has a closed form as follows:
\begin{equation*}
\sum_{k=0}^{n-1}[6k-1]\frac{(aq^{-1},q^{-1}/a,q^{-1};q^3)_k}
{(aq^{3},q^{3}/a,q^{3};q^3)_k}q^{3k}
=[3n-2][3n-4]\frac{(aq^{2},q^{2}/a,q^{-1};q^3)_{n-1} }
{(aq^{3},q^{3}/a,q^{3};q^3)_{n-1} },
\end{equation*}
which can be easily proved by induction on $n$.
The $a=1$ case implies that when $d=3$, the congruence \eqref{eq:main-2}
modulo $\Phi_n(q)^3$ is still true for $n\equiv 1\pmod{3}$ and $n>1$.

%%%%%%%%%%%%%%%%%%%%%%%%%%%%%%%%%%%%%%%%%%%%%%%%%%%%%%%%%%%%%%%%%%%%%%%%%%%%%%%%%%%%%%%%%%%%%%%%%%%%%%%%%%%%%%%%%%%%%%%%%%%%%%%%%%%%%%%%%%%%%%%%%%%%%%%%%%%%%%%%%%%%%%%%%%%%%%%%%%%%%%%%%%%%%%%%%%%%%%%%%%%%%%%%%%%%%%%%%%
\section{Proof of Theorem \ref{thm:main-3}}\label{sec:prth3}
By induction on $N$, we can easily prove that for $N>1$,
\begin{align}
\sum_{k=0}^{N-1}\frac{(q^{-1};q^2)_k^2}{(q^2;q^2)_k^2} q^{2k}
&=\frac{(q;q^2)_{N-1}^2}{(q^2;q^2)_{N-1}^2}(2[2N-3]+q^{2N-2}) \notag\\[5pt]
&= \begin{bmatrix}2N-2\\N-1\end{bmatrix}^2
\frac{2[2N-3]+q^{2N-2}}{(-q;q)_{N-1}^4}. \label{eq:inducton}
\end{align}
Note that $\frac{1}{[N]}\begin{bmatrix}\begin{smallmatrix}2N-2\\
N-1\end{smallmatrix}\end{bmatrix}$ is the well-known $q$-Catalan
number which is a polynomial in $q$ (see \cite{FH}). Thus
$[N]$ divides
$\begin{bmatrix}\begin{smallmatrix}2N-2\\N-1\end{smallmatrix}\end{bmatrix}$.
Moreover, it is easy to see that $[N]=\frac{1-q^N}{1-q}$ is relatively
prime to $(-q;q)_{N-1}$ for odd $N$. We conclude that \eqref{eq:main-3-1}
holds by taking $N=n$ in \eqref{eq:inducton}.

Letting $N=(n+3)/2$ in \eqref{eq:inducton}, we obtain
\begin{align*}
\sum_{k=0}^{\frac{n+1}{2}}\frac{(q^{-1};q^2)_k^2}{(q^2;q^2)_k^2} q^{2k}
=\begin{bmatrix}n+1\\(n+1)/2\end{bmatrix}^2
\frac{2[n]+q^{n+1}}{(-q;q)_{(n+1)/2}^4}.
\end{align*}
It is clear that
\begin{align*}
\frac{[(n+1)/2]}{[n]}\begin{bmatrix}n\\(n-1)/2\end{bmatrix}
=\begin{bmatrix}n-1\\(n-1)/2\end{bmatrix}
\end{align*}
is a polynomial in $q$. Since the polynomials $[(n+1)/2]$ and $[n]$ are
relatively prime, we deduce that
$\begin{bmatrix}\begin{smallmatrix}n\\(n-1)/2\end{smallmatrix}\end{bmatrix}$
is divisible by $[n]$, and so is
$\begin{bmatrix}\begin{smallmatrix}n+1\\(n+1)/2\end{smallmatrix}\end{bmatrix}
=(1+q^{(n+1)/2})\begin{bmatrix}
\begin{smallmatrix}n\\(n-1)/2\end{smallmatrix}\end{bmatrix}$.
The proof of \eqref{eq:main-3-2} then follows from the fact that $[n]$
is relatively prime to $(-q;q)_{(n+1)/2}$.

%%%%%%%%%%%%%%%%%%%%%%%%%%%%%%%%%%%%%%%%%%%%%%%%%%%%%%%%%%%%%%%%%%%%%%%%%%%%%%%%%%%%%%%%%%%%%%%%%%%%%%%%%%%%%%%%%%%%%%%%%%%%%%%%%%%%%%%%%%%%%%%%%%%%%%%%%%%%%%%%%%%%%%%%%%%%%%%%%%%%%%%%%%%%%%%%%%%%%%%%%%%%%%%%%%%%%%%%%%
\section{Proof of Theorem \ref{thm:main-4}}\label{sec:prth4}
By induction on $n$, we can prove that for $n\geqslant 1$
\begin{equation}
\sum_{k=0}^{n-1}\frac{(q^{-1},q^{-2};q^3)_k}{(q^3;q^3)_k^2} q^{3k}
=\frac{(2+q^{3n}-q-q^2-q^{3n-3})(q,q^{2};q^3)_{n-1} }
{(1-q)(1-q^2) (q^3;q^3)_{n-1}^2}.  \label{induction-2}
\end{equation}
We now assume that $n\geqslant 4$ and $\gcd(n,3)=1$.
If $n\equiv 1\pmod{3}$, then $(q;q^3)_{n-1}$ contains the factor $1-q^n$
and $(q^2;q^3)_{n-1}$ contains the factor $1-q^{2n}$,
and therefore $(q,q^{2};q^3)_{n-1}$ is divisible  by $\Phi_n(q)^2$.
If $n\equiv 2\pmod{3}$, then
$(q;q^3)_{n-1}$ contains $1-q^{2n}$ and $(q^2;q^3)_{n-1}$ contains $1-q^{n}$,
and $(q,q^{2};q^3)_{n-1}$ is also divisible  by $\Phi_n(q)^2$.
Clearly, the denominator of the right-hand side of
\eqref{induction-2} is relatively prime to $\Phi_n(q)$.
This completes the proof.

%%%%%%%%%%%%%%%%%%%%%%%%%%%%%%%%%%%%%%%%%%%%%%%%%%%%%%%%%%%%%%%%%%%%%%%%%%%%%%%%%%%%%%%%%%%%%%%%%%%%%%%%%%%%%%%%%%%%%%%%%%%%%%%%%%%%%%%%%%%%%%%%%%%%%%%%%%%%%%%%%%%%%%%%%%%%%%%%%%%%%%%%%%%%%%%%%%%%%%%%%%%%%%%%%%%%%%%%%%%%%%%%%%
\section{More congruences modulo $\Phi_n(q)^2$}\label{sec:more}

The first author \cite[Theorem 1.4]{Guo-m=d} proved that for $n>1$,
\begin{align}
\sum_{k=0}^{n-1}\frac{(q,q,q^4;q^6)_k}{(q^6;q^6)_k^3} q^{6k}
&\equiv 0 \pmod{\Phi_n(q)^2}\quad\text{if $n\equiv 5\pmod{6}$,}
\label{main-6-5}\\\intertext{and}
\sum_{k=0}^{n-1}\frac{(q^{-1}, q^{-1},q^{-4};q^6)_k}{(q^6;q^6)_k^3} q^{6k}
&\equiv 0 \pmod{\Phi_n(q)^2} \quad\text{if $n\equiv 1\pmod{6}$}.
\label{main-6-1}
\end{align}
Here we give generalizations of the $q$-congruences \eqref{main-6-5} and
\eqref{main-6-1} as follows.
\begin{theorem} Let $d\geqslant 3$ and let $r$ be a nonzero integer with
$|r|<d$ and $2r\ne\pm d$.
Let $n>1$ be an integer with $n\geqslant d-r$. Then
\begin{equation}
\sum_{k=0}^{n-1}\frac{(q^r,q^r,q^{d-2r};q^d)_k}
{(q^d;q^d)_k^3 }q^{dk} \equiv 0\pmod{\Phi_n(q)^2}\quad
\text{for  $n\equiv -r\pmod{d}$}. \label{eq:more-1}
\end{equation}
\end{theorem}
\begin{proof}The proof is similar to that of Theorem~\ref{thm:main-1}
(or \cite[Theorem 1.4]{Guo-m=d}). Here we merely give
the parametric generalization of \eqref{eq:more-1}:
\begin{equation*}
\sum_{k=0}^{n-1}\frac{(a^{d-1}q^r,a^{1-d}q^r,q^{d-2r};q^d)_k}
{(a^{d-2}q^d,a^{2-d}q^d,q^d;q^d)_k}q^{dk} \equiv 0\pmod{(1-aq^n)(a-q^n)}.\qedhere
\end{equation*}
\end{proof}

Note that when $d=4$ and $r=1$ the $q$-congruence \eqref{eq:more-1} was originally conjectured in
\cite[Conjecture 5.5]{GuoZu}.

\begin{theorem} Let $d\geqslant 3$ and let $0<r<d$ with $2r\ne d$.
Let $n\geqslant d+r$ be an integer. Then
\begin{equation}
\sum_{k=0}^{n-1}\frac{(q^{-r},q^{-r},q^{2r-d};q^d)_k}
{(q^d;q^d)_k^3 }q^{dk} \equiv 0\pmod{\Phi_n(q)^2}\quad
\text{for $n\equiv r\pmod{d}$}. \label{eq:more-2}
\end{equation}
\end{theorem}
\begin{proof}This time the parametric generalization of \eqref{eq:more-2}
is as follows:
\begin{equation*}
\sum_{k=0}^{n-1}\frac{(a^{d-1}q^{-r},a^{1-d}q^{-r},q^{2r-d};q^d)_k}
{(a^{d-2}q^d,a^{2-d}q^d,q^d;q^d)_k}q^{dk} \equiv 0\pmod{(1-aq^n)(a-q^n)}.\qedhere
\end{equation*}
\end{proof}

\section{Concluding remarks and open problems}\label{sec:rems}
The creative microscoping method used to prove Theorems~\ref{thm:main-1} and
\ref{thm:main-2} can be used to prove many other $q$-congruences
(see \cite{Guo-par,Guo-m=d,GS,GuoZu,GuoZu2}).
We also learned that this method has already caught
the interests of Gorodetsky~\cite{Gorodetsky}, Guillera~\cite{Guillera3}
and Straub~\cite{Straub}.
However, to the best of our knowledge,
the (creative) method of adding extra parameters can only be used to prove $q$-congruences modulo $\Phi_n(q)^3$ or
$\Phi_n(q)^2$ but not those modulo $\Phi_n(q)^4$ or higher powers of $\Phi_n(q)$.
The following conjectural refinements of Theorems~\ref{thm:main-1} and
\ref{thm:main-2} seem to be rather challenging to prove.
\begin{conjecture}
Let $d\geqslant 5$ be an odd integer. Then
\begin{equation*}
\sum_{k=0}^{n-1}[2dk+1]\frac{(q;q^d)_k^d}{(q^d;q^d)_k^d}q^{\frac{d(d-3)k}{2}}
\equiv
\begin{cases} 0\pmod{\Phi_n(q)^3}, &\text{if $n\equiv -1\pmod{d}$,}\\[5pt]
0\pmod{\Phi_n(q)^4}, &\text{if $n\equiv -\frac{1}{2}\pmod{d}$.}
\end{cases}
\end{equation*}
\end{conjecture}

\begin{conjecture}
Let $d\geqslant 5$ be an odd integer and let $n>1$. Then
\begin{equation*}
\sum_{k=0}^{n-1}[2dk-1]\frac{(q^{-1};q^d)_k^d}{(q^d;q^d)_k^d}q^{\frac{d(d-1)k}{2}}
\equiv
\begin{cases} 0\pmod{\Phi_n(q)^3}, &\text{if $n\equiv 1\pmod{d}$,}\\[5pt]
0\pmod{\Phi_n(q)^4}, &\text{if $n\equiv \frac{1}{2}\pmod{d}$.}
\end{cases}
\end{equation*}
\end{conjecture}

The first author \cite[Theorem 1.1]{Guo-m=d} proved that for
$d\geqslant 3$ and $n\equiv -1\pmod{d}$
\begin{equation}
\sum_{k=0}^{n-1}\frac{(q;q^d)_k^d}{(q^d;q^d)_k^d} q^{dk} \equiv
0 \pmod{\Phi_n(q)^2};  \label{eq:d-1}
\end{equation}
and that for $n,d\geqslant 2$ and $n\equiv 1\pmod{d}$
\begin{equation}
\sum_{k=0}^{n-1}\frac{(q^{-1};q^d)_k^d}{(q^d;q^d)_k^d} q^{dk} \equiv
0 \pmod{\Phi_n(q)^2}.  \label{eq:d-2}
\end{equation}
These two $q$-congruences were originally conjectured by the first author and
Zudilin \cite[Conjectures 5.3 and 5.4]{GuoZu}.
Here we would like to make some similar conjectures on congruences
modulo $\Phi_n(q)^2$.

\begin{conjecture}\label{conj:123}
Let $d\geqslant 3$ and $n>1$ be integers with $n\equiv -1\pmod {d(d+1)/2}$.
Then
\begin{equation*}
\sum_{k=0}^{n-1}\frac{(q,q^2,\ldots,q^d;q^{d(d+1)/2})_k}
{(q^{d(d+1)/2};q^{d(d+1)/2})_k^d}q^{d(d+1)k/2}
\equiv 0\pmod{\Phi_n(q)^2}.
\end{equation*}
In particular, if $p\equiv -1\pmod {d(d+1)/2}$ is a prime and $m=d(d+1)/2$, then
\begin{equation*}
\sum_{k=0}^{p-1}
\frac{(\frac{1}{m})_k (\frac{2}{m})_k \cdots (\frac{d}{m})_k}{k!^d}
\equiv 0\pmod{p^2}.
\end{equation*}
\end{conjecture}

\begin{conjecture}\label{conj:-1-2-3}
Let $d\geqslant 2$ and $n>1$ be integers with $n\equiv 1\pmod {d(d+1)/2}$.
Then
\begin{equation*}
\sum_{k=0}^{n-1}\frac{(q^{-1},q^{-2},\ldots,q^{-d};q^{d(d+1)/2})_k}
{(q^{d(d+1)/2};q^{d(d+1)/2})_k^d}q^{d(d+1)k/2}
\equiv 0\pmod{\Phi_n(q)^2}.
\end{equation*}
In particular, if $p\equiv 1\pmod {d(d+1)/2}$ is a prime and $m=d(d+1)/2$, then
\begin{equation*}
\sum_{k=0}^{p-1}
\frac{(-\frac{1}{m})_k (-\frac{2}{m})_k \cdots (-\frac{d}{m})_k}{k!^d}
\equiv 0\pmod{p^2}.
\end{equation*}
\end{conjecture}
We should concede that we are not able to prove Conjectures~\ref{conj:123}
and \ref{conj:-1-2-3} even for $d=3$ (we are only capable to deal with
the modulus $\Phi_n(q)$ case).
Note that Conjecture~\ref{conj:-1-2-3} is true for $d=2$ by Theorem~\ref{thm:main-4}.

\begin{conjecture}\label{conj:135}
Let $d\geqslant 3$ and $n>1$ be integers with $n\equiv -1\pmod {d^2}$. Then
\begin{equation*}
\sum_{k=0}^{n-1}\frac{(q,q^3,\ldots,q^{2d-1};q^{d^2})_k}
{(q^{d^2};q^{d^2})_k^d}q^{d^2 k} \equiv 0\pmod{\Phi_n(q)^2}.
\end{equation*}
In particular, if $p\equiv -1\pmod {d^2}$ is a prime, then
\begin{equation*}
\sum_{k=0}^{p-1}
\frac{(\frac{1}{d^2})_k (\frac{3}{d^2})_k \cdots (\frac{2d-1}{d^2})_k}{k!^d}
\equiv 0\pmod{p^2}.
\end{equation*}
\end{conjecture}

\begin{conjecture}\label{conj:-1-3-5}
Let $d\geqslant 2$ and $n>1$ be integers with $n\equiv 1\pmod {d^2}$. Then
\begin{equation*}
\sum_{k=0}^{n-1}\frac{(q^{-1},q^{-3},\ldots,q^{-2d+1};q^{d^2})_k}
{(q^{d^2};q^{d^2})_k^d}q^{d^2k} \equiv 0\pmod{\Phi_n(q)^2}.
\end{equation*}
In particular, if $p\equiv 1\pmod {d^2}$ is a prime, then
\begin{equation*}
\sum_{k=0}^{p-1}
\frac{(-\frac{1}{d^2})_k (-\frac{3}{d^2})_k \cdots (-\frac{2d-1}{d^2})_k}{k!^d}
\equiv 0\pmod{p^2}.
\end{equation*}
\end{conjecture}

Using the following identity
\begin{equation*}
\sum_{k=0}^{n-1}\frac{(q^{-1},q^{-3};q^4)_k}{(q^4;q^4)_k^2} q^{4k}
=\frac{(2+q^{4n}-q-q^3-q^{4n-4})(q,q^{3};q^4)_{n-1} }
{(1-q)(1-q^3) (q^4;q^4)_{n-1}^2},
\end{equation*}
we can easily prove that Conjecture \ref{conj:-1-3-5} is true for $d=2$.

It seems that Conjectures \ref{conj:123} and \ref{conj:-1-2-3}
can be further generalized as follows.
\begin{conjecture}\label{conj:last-1}
Let $d$ and $r$ be positive integers with $dr\geqslant 3$.
Let $n>1$ be an integer with $n\equiv -1\pmod {d(d+1)r/2}$. Then
\begin{equation*}
\sum_{k=0}^{n-1}\frac{(q,q^2,\ldots,q^d;q^{d(d+1)r/2})_k^r}
{(q^{d(d+1)r/2};q^{d(d+1)r/2})_k^{dr}}q^{d(d+1)rk/2}
\equiv 0\pmod{\Phi_n(q)^2}.
\end{equation*}
\end{conjecture}

\begin{conjecture}\label{conj:last-2}
Let $d$ and $r$ be positive integers with $dr\geqslant 2$.
Let $n>1$ be an integer with $n\equiv 1\pmod {d(d+1)r/2}$. Then
\begin{equation*}
\sum_{k=0}^{n-1}\frac{(q^{-1},q^{-2},\ldots,q^{-d};q^{d(d+1)r/2})_k^r}
{(q^{d(d+1)r/2};q^{d(d+1)r/2})_k^{dr}}q^{d(d+1)rk/2}
\equiv 0\pmod{\Phi_n(q)^2}.
\end{equation*}
\end{conjecture}

Likewise, Conjectures \ref{conj:135} and \ref{conj:-1-3-5}
can be further generalized as follows.
\begin{conjecture}\label{conj:last-3}
Let $d$ and $r$ be positive integers with $dr\geqslant 3$.
Let $n>1$ be an integer with $n\equiv -1\pmod {d^2r}$. Then
\begin{equation*}
\sum_{k=0}^{n-1}\frac{(q,q^3,\ldots,q^{2d-1};q^{d^2r})_k^r}
{(q^{d^2r};q^{d^2r})_k^{dr}}q^{d^2r k}
\equiv 0\pmod{\Phi_n(q)^2}.
\end{equation*}
\end{conjecture}

\begin{conjecture}\label{conj:last-4}
Let $d$ and $r$ be positive integers with $dr\geqslant 2$.
Let $n>1$ be an integer with $n\equiv 1\pmod {d^2r}$. Then
\begin{equation*}
\sum_{k=0}^{n-1}\frac{(q^{-1},q^{-3},\ldots,q^{-2d+1};q^{d^2r})_k^r}
{(q^{d^2r};q^{d^2r})_k^{dr}}q^{d^2rk}
\equiv 0\pmod{\Phi_n(q)^2}.
\end{equation*}
\end{conjecture}
Finally, we point out that Conjectures~\ref{conj:last-1}--\ref{conj:last-4}
are clearly true
for $d=1$ by the $d\mapsto r$ cases of \eqref{eq:d-1} and \eqref{eq:d-2}.

\vskip 5mm
\noindent{\bf Acknowledgments.} We thank Wadim Zudilin for helpful comments
on a previous version of this paper. We further thank the referees for their
careful reading of the manuscript; their comments led to improvements of
the exposition.

\end{document}